\documentclass[11pt]{amsart}
\usepackage{geometry}
\usepackage{hyperref,MnSymbol,mathrsfs,amsmath,amsthm,tikz}

\usetikzlibrary{
  knots,
  hobby,
  decorations.pathreplacing,
  shapes.geometric,
  calc,
  decorations.markings
}
\usepgfmodule{decorations}
\usepackage{float, braids, setspace, tikz-cd}

\usepackage{wrapfig}
\usepackage{lipsum} 
\usepackage{pb-diagram}
\theoremstyle{plain}
\theoremstyle{definition}
\newtheorem{Theorem}{Theorem}[section]
\newtheorem{quest}[Theorem]{Question}

\newtheorem{Lemma}[Theorem]{Lemma}
\newtheorem{Proposition}[Theorem]{Proposition}

\newtheorem{Definition}[Theorem]{Definition}

\numberwithin{equation}{section}
\DeclareMathAlphabet\mathbb{U}{msb}{m}{n}

\DeclareMathAlphabet\mathbb{U}{msb}{m}{n}

\newcommand{\R}{\mathbb R}

\newcommand{\Z}{\mathbb Z}

\newcommand{\K}{\pi_1 (K)}
\newcommand{\AB}{\{a,b\}}

\begin{document}

\providecommand{\keywords}[1]
{
  \small	
  \textbf{\textit{Keywords---}} #1
}
\title{Geodesic words in central extension of the Klein Bottle group}

\author{Oleg Chistov}
\author{Ruslan Magdiev}

\address{Laboratory of Continuous Mathematical Education (School 564 of St. Petersburg), nab. Obvodnogo kanala 143, Saint Petersburg, Russia}
\email{oleguszte@gmail.com}
\address{ St. Petersburg State University, 14th Line, 29b, Saint Petersburg, 199178 Russia}
\email{rus.magdy@mail.ru}

\begin{abstract}
    In this work, we give a complete description of the language of geodesic words for a central extension of the fundamental group of the Klein Bottle with respect to the standard two-element generating set. Besides, we prove that there are no dead ends in the group. Finally, we give a new sufficient condition for the absence of dead ends.
\end{abstract}


\maketitle

\section{Introduction}

\noindent We study combinatorial group theory aspects of the group given by the following presentation:

$$ cK = \langle a,b \mid [aba^{-1}b,a] = [b, aba^{-1}b] = 1 \rangle $$

The research on this group is motivated by M. Shapiro's works
\cite{SHP,PSL}, where one can find the idea of studying geometries of finitely generated groups whose elements can be represented as paths on the Cayley graphs of much simpler groups. More precisely, M. Shapiro considered groups whose presentation complexes are orientable surfaces and then consider groups whose elements can be represented by equivalence classes of paths in the Cayley graph of the original groups.
The most demonstrative examples of this method of studying a group are works dedicated to the discrete Heisenberg group. In \cite{RUS, SHP, InfiniteG}
there are examples of using this construction with "drawing" elements.
A peculiarity of these works is that nobody has studied the case when presentation complex is non-orientable surface, which raises a natural question.

\begin{quest}
If we apply M. Shapiro's construction for the case of non-orientable surface as presentation complex, how drastically combinatorial and geometric properties of a group does change?
\end{quest}

To answer this question, the authors of the paper consider the simplest example of an infinite group whose presentation complex is a non-orientable surface - the Klein bottle group.
More precisely, $cK$, a group was presented earlier, is the central extension of $\pi_1(K)$ with the following presentation and short exact sequence:

$$ \pi_1 (K) = \langle a,b \mid aba^{-1}b \rangle; $$

$$ 1 \longrightarrow \langle t \rangle \overset{\alpha}\longrightarrow cK \overset{\pi}\longrightarrow \pi_1 (K) \longrightarrow 1,$$
\\
where $\alpha(t) = aba^{-1}b$ and $\pi$ is the canonical projection onto the quotient by the normal subgroup generated by $aba^{-1}b$. Below we write $\{a,b\} = aba^{-1}b$

Actually, there are two non-isomorphic groups that can be obtained as central extensions \cite{D03}, but the second one is $\pi_1 (K) \times \Z$ which is rather not interesting in terms of geometry.

\section{Main results}

The main result of this work is a new method that is similar to the one described in M.Shapiro's work \cite{SHP}.
The result consists of the following facts: there are no dead end elements in $ cK $ relative to the standard generating set, there is an explicit description of all geodesic words and their geodesic continuations in terms of "basic moves"(Section 5), and there is sufficient condition on the absence of dead ends in terms of the formal language of geodesic words. All together these facts give one an ability to study geometric aspects of $ cK $ in terms of projection on $\pi(K)$ and ideas written in Sections 3 and 4.

\vspace{3mm}

\noindent {\bf Theorem 1}.
In $cK$ there are no dead end elements. 

\vspace{3mm}

\noindent {\bf Theorem 2}.
Any geodesic representative of element $(k,m,n)$ can be obtained from a standard geodesic representative by finitely many applying of {\it basic moves}.

\vspace{3mm}

\noindent {\bf Theorem 3}.

The following statements hold:

\begin{enumerate}

    \item[1.] Any geodesic representative of $ (k, m, n) $ can be continued with $a$ and $b^{-1}$, where $ k < 0 \And | k | > m $.
    
    \item[2.] Any geodesic representative of $ (k, m, n) $ can be continued with $a$ and $b$, where $ k < 0 \And | k | \leqslant m $, $n \mod 2 = 0$.
    
    \item[3.] Any geodesic representative of $ (k, m, n) $ can be continued with $a$ and $b^{-1}$, where $ k < 0 \And | k | \leqslant m $, $n \mod 2 \neq 0$.
    
    \item[4.] Any geodesic representative of $ (k, m, n) $ can be continued with $a$ and $b$, where $ k > 0$.

\end{enumerate}

\vspace{3mm}

\noindent {\bf Theorem 4}.
Consider a finitely generated even given group $ G = \langle S \mid R \rangle $ with a fixed generating set $ S $. If $ St (G, S) $ consists of geodesic words for each element of the group and each element of the language is represented as a prefix of another element, then there are no dead ends in $ G $ with respect to $ S $.

\section{Basic definitions}

\vspace{0.3cm}

The basic definitions from geometric group theory and general terminology can be found in the following books \cite{OH, CL}, but the authors want to put emphasis on some of them below.

\vspace{0.2cm}

\subsection{Combinatorial and geometric group theory}

Let $S$ be a set of letters. By $F(S)$ we denote the free group generated by $S$. For any $w\in F(S)$ the length of the word is $ |w| $.  By the length of an element $ g \in G $ we mean the length of the shortest word representing  $ g $, and denote it by $ l(g). $

\begin{Definition}
Word $w \in F(S)$ is called geodesic if $l(w) = |w|$. I denote the set of all geodesic words in the group $G$ by $\Gamma(S,G) \subseteq F(S)$
\end{Definition}
\begin{Definition}
Element $g \in G$ is called dead end if ${\rm depth}(g)>0$ i.e $l(gs) \leq l(g)$ for all $s \in S \cup S^{-1}$. I denote the set of all dead end words $G$ by $\Delta(G) \subseteq G$. Word $w$ is called dead end if it is geodesic and defines an element: $l(\overline{w}s) \leq l(\overline{w}) = |w|$ for all $s \in S \cup S^{-1}$.
\end{Definition}

\begin{Definition}
The normal form of an element $ g \in  G $ is called an isomorphism from group $ G $ to $ \Z^n $ with a non-standard operation of multiplication of tuples for some $ n \in \Z^+ $.
\end{Definition}

\begin{Definition}
The Dehn area of word $ w \in F(S) $ that is identity in group $ G = \langle S \mid R\rangle, |S|+|R|<\infty$ is the minimum number of applications of relations, insertions, and reductions that are needed to bring a word to an empty word.
\end{Definition}

\section{Model of $cK$ as $\pi_1 (K)$}

\vspace{0.3cm}

Since there is the canonical projection of $ cK $ on $ \pi_1 (K) $, each element of the group $ \pi_1 (K) $ is an equivalence class of elements in $ cK $. More precisely, two elements $ g, h \in cK $ are equivalent, if $ gh^{- 1} \in \langle \{ a, b \} \rangle $.

Note that the canonical projection induces the projection of the Cayley graph $ cK $ onto the Cayley graph $ \K $. Then we can project polygonal lines from the Cayley graph of $ cK $ into polylines from $ \K $. It is important to mark that polygonal lines do not change in the way of notion: each polygonal line on the Cayley graph of $ cK $ is some word from elements in the alphabet of a generating set, but since the canonical projection preserves the system of generators, then each polyline from $ cK $ corresponds to a unique polygonal line on the Cayley graph of $ \K $. Then the polyline on the Cayley graph $ \K $ uniquely defines an element of the group.

\begin{figure}
\center{\includegraphics[scale=1.3]{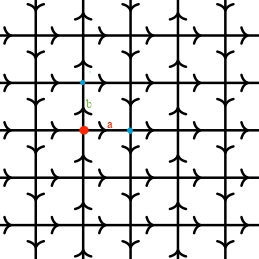}}
\caption{\footnotesize Cayley graph of group $\K$. Red dot shows 1}
\end{figure}

Up to this point, no specific characteristics of the structure of Cayley graphs were used. This means that one can consider any groups where one group is a central extension of the other group instead of $ \K $ and $ cK $ 

Now let's turn directly to the structure of the $ \K $ Cayley graph. It is a well-known fact that $ \K $ is a semidirect product of $ \Z $ with $ \Z $, which means that each element can be represented as a pair of integers. This pair will be called the normal form of the element. What will the normal form represent? Recall that the Klein bottle is topologically covered by a plane. It is not difficult to understand that the plane will be tiling of the Klein bottle with a sweep.
If one glue disks into each square as indicated in the group relations, then a Cayley complex of $ \K $ is obtained. It is an integer lattice with a special edge orientation that corresponds to the Cayley graph $ \K $ with respect to the standard generating set. It is clear that each pair of numbers $ (m, n) $ can be associated with a polyline $ b^m a^n $. Each such polygonal chain is uniquely defined by its end, therefore this word implements the normal form.
Let's see what element in $ cK $ the word $ b^m a^n $ defines. This will be an element defined by the $ \AB^0 b^m a^n $ word. Now when considering any other polyline $ w $, the exponent at $ \AB $ will be equal to the Dehn area of the broken line $ w*a^{-n}b^{-m} $. Thus, one can encode each element of the group $ cK $ with three integers $ (k, m, n) $, where $ m, n $ are the coordinates of the end of the polyline on the lattice $ \K $, and $ k $ is the Dehn area of a closed broken line obtained from the given one by assigning $ a$.
How can one geometrically see the Dehn square? Let's take a look at the geodesic representatives $ \{ a, b \} $. Their length will be equal to 4 and they will walk around the cell in a certain order. Since $ \{ a, b \} $ has Dehn's area equal to 1, then the order of traversal given by the geodesic representatives is considered "positive" and one can assume that for each cell there is definitely such an order: clockwise or counterclockwise.

\begin{figure}
\center{\includegraphics[scale=1.3]{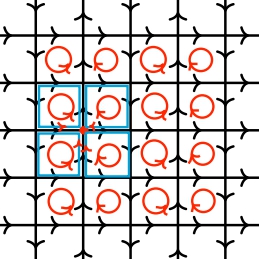}}
\caption{\footnotesize Caley graph of group $\K$ with local orientation in squares}
\end{figure}

\subsection{Normal form}
Based on the entered normal form formula one can change the derivation of the method of multiplying the normal forms of the elements of this group. Let's take elements $g$ and $h$ with normal forms $(k_1,m_1,n_1)$ and $(k_2,m_2,n_2)$ respectively. The normal form of element $gh$ will be equal: 
$$(k_1 + k_2 + m_2(n_1\mod{2}), m_1 + m_2(-1)^{n_1}, n_1 + n_2) $$

The presence of a normal form and a formula for their multiplication allows us to introduce isometries of the Cayley graph $cK$ which will narrow down the number of cases analyzed in lemmas and proofs:
$$(k,m,n) \longmapsto (k,m,-n) \qquad (k,m,n) \longmapsto (-k,-m,-n) $$
These transformations are generated by two transformation of the generating set ($a \mapsto a^{-1}$ and $a,b \mapsto a^{-1}, b^{-1}$ that preserve "being-geodesic". They allow us to further assume that we are considering only the cases $n,m\geqslant 0$ unless otherwise stated.

There is a list of several observations that are related to the model and normal form, formulated as lemmas: 

\begin{Lemma}
$a^2$ is in the center of $cK$.
\end{Lemma}

\begin{proof}
Let's take a look at how we can commute elements using relations.

\begin{center}
    $ab = \{ a,b \} b^{-1}a$\\
$a^{-1}b = \{ a,b \} b^{-1}a^{-1}$\\
$ab^{-1} = \{ a,b \}^{-1} ba$\\
$a^{-1}b ^{-1}= \{ a,b \}^{-1} ba^{-1}$\\

\end{center}

It is easy to see that when you try to commute the $a$ or $b$ with $a^2$ normal form doesn't change. 
\end{proof}

\begin{Lemma}\label{cut}
All closed contours in the model that are rectangles with an even width have a zero area.
\end{Lemma}

\begin{proof}
Note that by the previous lemma the contour $a^2 b a^{-2} b^{-1}$ will have a zero area. Any rectangle with an even width can be obtained as small 1x2 rectangles which are connected by some sections of polylines along the path. 
\end{proof}

\section{Auxiliary results}

In this section, the authors write out some general statements about the structure of geodesics and dead ends in groups, as well as lemmas that analyze certain classes of elements for dead ends.

\begin{Lemma} \label{EvenRelatorLemma}
Let $ G $ have relations of even length. Then $ l (gs) = l (g) \pm 1 $ for all $ s \in S \cup S ^ {- 1} $. In particular, there are no dead ends of depth 1 in $ G $.
\end{Lemma}
\begin{proof}
From the triangle inequality we have $ l (g) - 1 \leq l (gs) \leq l (g) + 1 $, so it suffices to prove that $ l (gs) \neq l (g). $ Suppose that the equality $ l (gs) = l (g) $, and we arrive at a contradiction. Let us choose a geodesic representative $ w \in F (S) $ for $ g $.
Note that the word $ ws $ is not geodesic, because $ l (gs) = l (g) = | w | <| ws | $, but $ [ws] = gs. $ Choose a geodesic representative $ u \in F (S) $ for $ gs $. Since the words $ u $ and $ ws $ define the same element $ G $, there is a path in the graph of words and relations between them. Note that passing along the edge does not change the parity of the word length, since $ G $ is evenly given. Thus, both numbers $ | u | $ and $ | ws | $ have the same parity. However, this is not the case, since $ | u | = l (gs) = l (g) = | w | $ and $ | ws | = | w | + 1 $. The resulting contradiction completes the proof.
\end{proof}

\begin{Lemma} \label{LastLetterLemma}
Let $ G $ have relations of even length. For the geodesic word $ w \in F(S) $ and the letter $ s \in S \cup S^{-1} $, the following conditions are equivalent:
\begin {enumerate}
\item $ ws $ is a geodesic word;
\item The element $ [w] \in G $ does not have a geodesic representative ending with the letter $ s^{-1} $.
\end{enumerate}
\end{Lemma}
\begin{proof}
Let $ ws \in \Gamma (S, G) $. Assume the opposite: let there exist a geodesic representative $ u \in [w] $, which ends with $ s^{-1} $.
Then since $ ws $ is geodesic, then we can replace the prefix $ w $ in it with $ u $. The length of the resulting word will not change, but we get that $ [ws] = [us^{-1}s] = [v] $, where $ v $ is the geodesic prefix $ w $ and $ | w | = | v | + 1 $. Moreover, $ | ws | = | w | + 1 = | v | + 2 $, which is impossible, since $ ws $ is a geodesic representative of its class, and we found a word with a shorter length in this class.

Now let the element $ [w] \in G $ have no geodesic representative ending with the letter $ s^{-1} $. Let us prove that $ ws \in \Gamma (S, G) $. Suppose the opposite. Let $ [ws] = [v] $, where $ v $ is a geodesic word, and $ |v| = |w| - 1 $ (see lemma \ref{EvenRelatorLemma}). Obviously, $ [w] = [vs^{-1}] $. Then $ l([vs^{-1}]) = l([w]) = |w| = (|w|-1)+1 =|vs^{-1}|$, that is, $ [w] $ has a geodesic representative ending in $ s^{-1} $. Contradiction with the condition. The lemma is proved.
\end{proof}

\begin{Lemma} \label{DEC}
The element $ g \in G $ is a dead end if and only if for each letter $ s \in S \cup S^{-1}$ $ g $ has a geodesic representative ending with this letter.
\end{Lemma}
\begin{proof}
Let us prove that $ g $ is a dead end. It is clear that if we consider some geodesic word $ w = vs $ and $ [w] = g $, then $ l([ws^{- 1}]) = l([v]) = |v| = |w|-1 = l([w]) - 1 $. Now, since we can find words from $ [w] $ ending with any letter of the alphabet, then $ l (gs)<l(g), \forall \, s \in S \cup S^{-1} $, which is the definition of a dead-end element.

Let $ g $ be a dead end. Suppose the opposite. Suppose we have no geodesic representatives $ g $ that end with $ s^{-1}$. Then, by Lemma \ref{LastLetterLemma}, if we assign $ s $ to any geodesic representative $ g $ on the right, then we again get a geodesic word. Since $ w $ and $ ws $ are geodesics, then $ l(gs) = |ws| = |w| + 1 = l(g) + 1 > l(g) $, which contradicts the fact that $ g $ is a dead-end element.
\end{proof}

\subsection{Dead-end elements and standard representatives} 
In this scientific work, it was proved that any element of the group $cK$ is not dead end. To prove this, several lemmas were introduced to show that elements of a certain kind are not dead ends. We also introduced the concept of a 'standard' element representative which is a kind of a typical element representative. As it turns out later, standard representatives are always geodesic. 

\noindent {\bf Definition}. Let's consider $(k,m,n)$, element from $cK$, where $m \geqslant 0, n \geqslant 0$  is a representative of the form $b^{k+m}ab^{k}a^{n-1}$ call a standard.

\noindent {\bf Definition}. The language of the standard representatives of the group $ G $ with respect to the generating set $ S $ is the following set of words:

$$St(G,S):= \{ w\in F(S) \mid w\, \text{ is a standard representative of } (k,m,n)\in G\} $$
\begin{Lemma}
All geodesic representatives of the element $ (k, 0,0) $ are cyclic shifts of the word $ w = b ^ {- k} a ^ {- 1} b ^ {- k} a $.
\end{Lemma}

According to \ref{cut}, it is clear that any geodesic presenter must have a minimum width because otherwise any local increase in width will either zero the locally accumulated area or leave it unchanged, but add the number of edges that need to be traversed at the beginning of the traversal of the negative area and at the end of the traversal. Hence the geodesic representatives $ (k,0,0) $ are closed contours of width 1, it is easy to see that $ w = b^{-k}a^{-1}b^{-k}a$ will be such a representative. It is also clear that with a cyclical shift of the word, we move our closed polygonal line along the contour. thereby covering all possible locations of the polygonal line relative to the origin.

This lemma implies that $ (k, 0,0) $ has no geodesic representative ending with the letter $ b ^ {- 1} $, which by \ref{DEC} means that $ (k, 0,0) $ is not a dead end and no prefixes of geodesic representatives of $ (k, 0,0) $ will be dead ends.

\begin{Lemma}
All elements of the form $ (k, m, n) $, where $ n \in \{ 0, 1 \}$ are not dead ends.
\end{Lemma}

\begin{proof}
It's obvious for $k<0$, so let's prove it for $k>0$. Note that is sufficient enough to prove that every word $a b^k a^{-1} b^{k+m}$ is geodesic. Let's do it by induction on $k+m$. Variation with $m=0$ is obvious, so it is going to be base of induction. Let's prove if $a b^k a^{-1} b^{k+m}$ is geodesic then $a b^k a^{-1} b^{k+m+1}$ is also geodesic. If it isn't then there have to be word $\omega$ that represent same element but at least $2$ letters less. Then if length of $a b^k a^{-1} b^{k+m}$ is $2k+m+2$ then $|\omega|$ $\leqslant$ $2k+m$. Also we need at least $m$ for going from $(k,0,0)$ to $(k,m,0)$. Then we have $2k$ letters to cover area $k$ which is impossible.
\end{proof}

\begin{Lemma} \label{lol}

Any element of the form $ (k, 0, n) $ is not a dead end.
\end{Lemma}
\begin{proof}
Consider the standard representative of the element: $ w = b^kab^ka^{n-1} $. It is clear that $ b ^ k a b ^ k $ is the prefix of the geodesic representative $ (k, 0,0) $, and it is clear that if $ b ^ k a b ^ k a $ is a geodesic word, then $ w $ is a geodesic word. Let us prove that $ b ^ k a b ^ k a $ is a geodesic representative. The length of this word is $ 2 * | k | + 2 $. Suppose there is an equivalent word of a shorter length. Then, by \ref{EvenRelatorLemma}, its length is less than $ 2 | k | $. In this word, one letter $ a $ will be used to shift the third coordinate of the normal form by one. Then $ 2 | k | -1 $ will be used to collect the negative area $ k $, which means that, according to the Dirichlet principle, we will definitely have an odd number of letters left to collect some amount of area, which is again impossible for even group assignments.

Now it is clear that the standard representatives of such elements are geodesic. Also, from each standard representative, you can go to another standard representative by multiplying the word by $ a $. This means that no such element is a dead end.

\end{proof}

\begin{Lemma}

An element of the form $ (k, m, n) $, where $ k <0 \And | k | \leqslant m $ are not dead ends.
\end{Lemma}
\begin{proof} Consider a standard representative: $ w = b ^ {k + m} ab ^ {k} a ^ {n-1} $. It is not difficult to see that in this case, the word will be a monotone broken line lying in the rectangle $ n $ x $ m $. Let's calculate its length: $ | w | = m + k + 1 -k + n-1 = m + n $, which is the minimum length of the polygonal line connecting opposite diagonal points of the rectangle with the lower-left vertex in $ (0,0) $. This means that such standard representatives are geodesic, and also from each standard representative, you can go to another standard representative by multiplying the word by $ a $. This means that no such element is a dead end.

\end{proof}

\begin{Lemma}

Any element of the form $ (k, m, n) $, where $ k <0 \And | k | \geqslant m $ is not a dead-end.
\end{Lemma}
\begin{proof}
We fix an element $ \tau \in cK $ with a normal form $ (k, m, n) $ satisfying the condition. Consider a standard representative for an element. Let us prove that each of its prefixes is a geodesic word. It is clear that as long as the prefixes coincide with the prefixes of the geodesic representatives $ (k, 0,0) $, the geodesicity is preserved. Now let us prove that $ b ^ {k + m} ab ^ {k} a $ is geodesic. We proved this in \ref{lol}. Then the further proof works as in \ref{lol}.
\end{proof}

The following result actually generalizes the construction with standard representatives to an arbitrary finitely generated group, finding a sufficient condition for the absence of dead ends in a group with respect to a fixed generating set.

\begin{Lemma}
Any element of the form $ (k, m, n) $, where $ k> 0 \And m> 0 $ is not a dead end.
\end{Lemma}
\begin{proof}
Consider the standard element representative $ w = b^{k+m}ab^{s}a^{n-1} $. It is clear that $ b^kab^s $ is the prefix of the geodesic representative $ (k,0,0) $, then it is geodesic. Now I will prove that we can multiply by $ a $ and the word still will be geodesic. Let's prove that $ b^{m + k}ab^ka$ is a geodesic. We proved this in \ref{lol}. Then it is clear that all the standard representatives of these elements are geodesic. Then these elements are not dead ends.

\end{proof}

\begin{Proposition}

Every element of $St(cK, \{a,b\})$ is geodesic

\end{Proposition}

\begin{proof}

It was proved in lemmas recursively.

\end{proof}

\begin{Theorem}\label{kek}
Consider a finitely generated even given group $ G = \langle S \mid R \rangle $ with a fixed generating set $ S $. If $ St (G, S) $ consists of geodesic words for each element of the group and each element of the language is represented as a prefix of another element, then there are no dead ends in $ G $ with respect to $ S $.
\end{Theorem}

\begin{proof}
It is clear that if each element of the group has a geodesic representative and at the same time it is a prefix of another standard geodesic representative, then the original element is no longer a dead end.
\end{proof}

\subsection{Geodesic continuity and basic moves}

The following classification theorem is introduced to understand how we can continue geodesics:

\begin{Theorem}

The following statements hold:

\begin{enumerate}

    \item[1.] Any geodesic representative of $ (k, m, n) $ can be continued with $a$ and $b^{-1}$, where $ k < 0 \And | k | > m $.
    
    \item[2.] Any geodesic representative of $ (k, m, n) $ can be continued with $a$ and $b$, where $ k < 0 \And | k | \leqslant m $, $n \mod 2 = 0$.
    
    \item[3.] Any geodesic representative of $ (k, m, n) $ can be continued with $a$ and $b^{-1}$, where $ k < 0 \And | k | \leqslant m $, $n \mod 2 \neq 0$.
    
    \item[4.] Any geodesic representative of $ (k, m, n) $ can be continued with $a$ and $b$, where $ k > 0$.

\end{enumerate}
\end{Theorem}

\begin{center}

\begin{figure}
\center{\includegraphics[scale = 0.2]{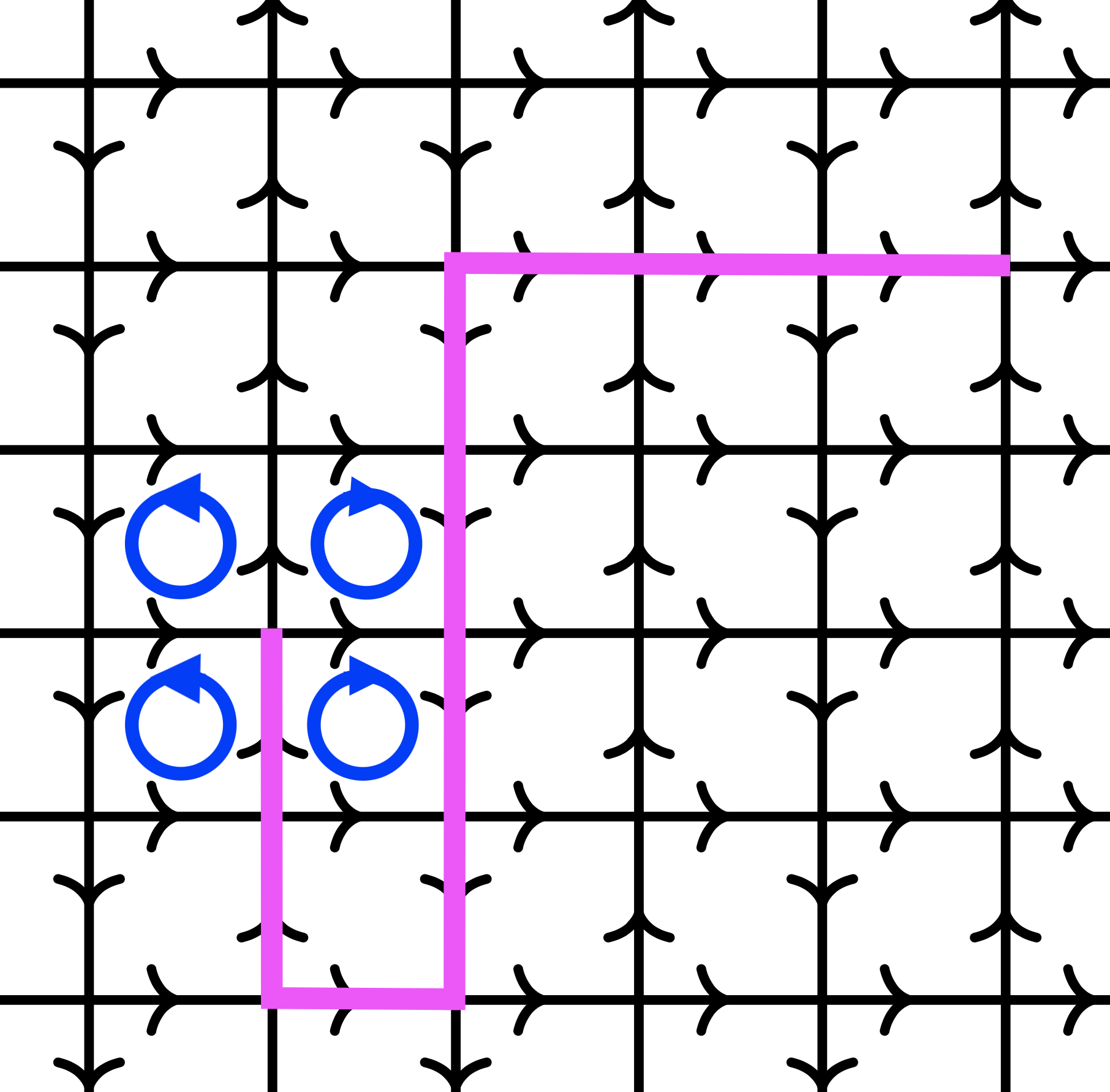}}
\caption{\footnotesize Standard representative for $(-4,2,4)$ }
\end{figure}

\begin{figure}
\center{\includegraphics[scale = 0.2]{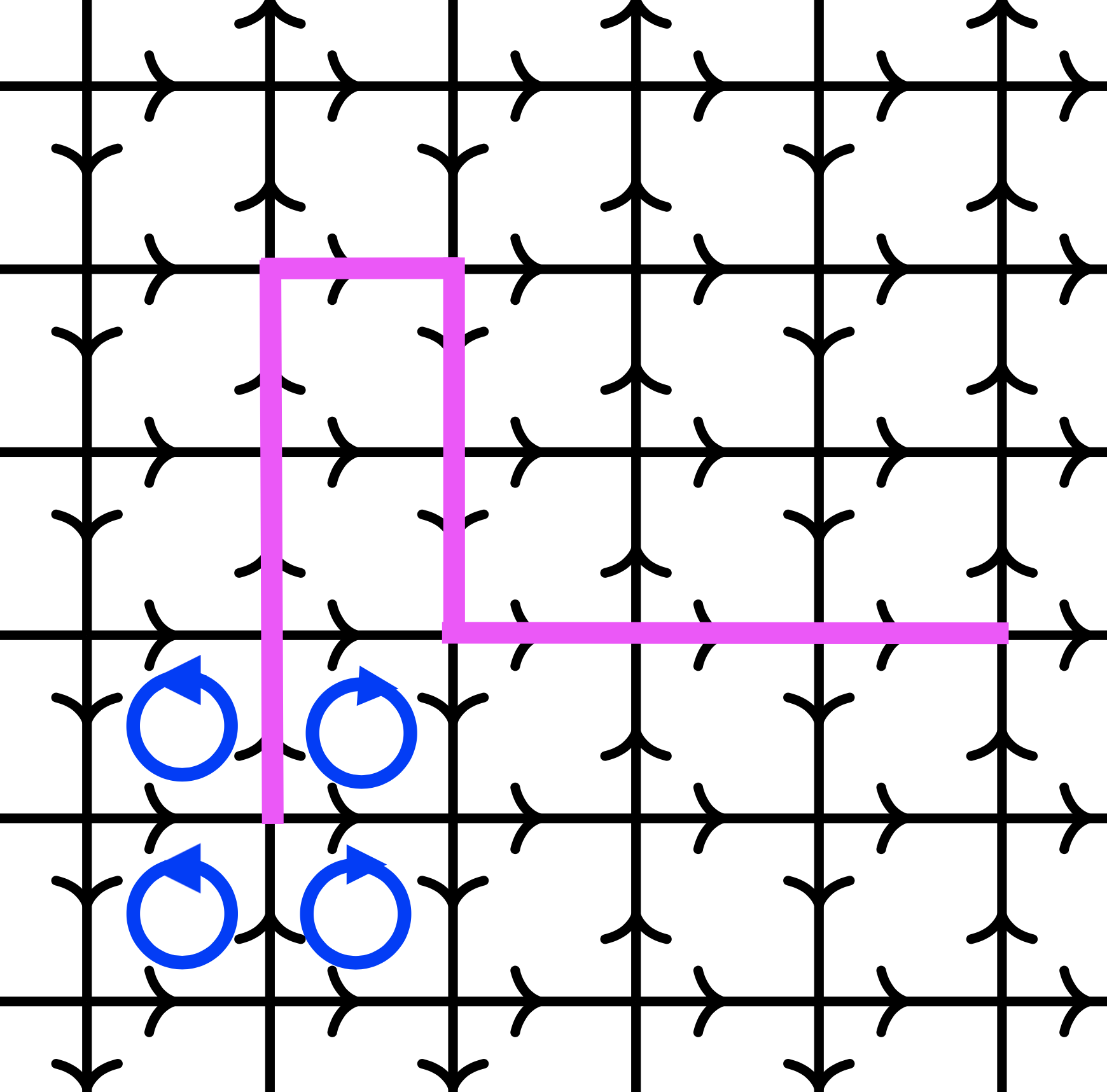}}
\caption{\footnotesize Standard representative for $(2,1,4)$ }
\end{figure}

\end{center}

\begin{proof}
The proof is based on the estimation of length with standard representatives, using technical lemmas in the section dedicated to combinatorial and geometric group theory.

\end{proof}

The concept of standard geodesic representative provides us with the length formula of an element. Also, it helps to describe other geodesic representatives, but to do it I introduce three basic moves:

\begin{Definition}\label{kok}
Let $w\in F(S)$ be geodesic representative of a given element. Basic moves are one of the following transformations:

\begin{enumerate}
    \item[1. ] Even castling. If there is subword $b^{\pm}a^2$ or $a^2b^{\pm}$. Then one can change it with $a^2 b^{\mp}$ or $b^{\mp}a^2$
    
    \item[2. ] Detowering. If there is subword $w_0 = b^{\pm x} a b^{\mp x}$  where $x\leqslant0$ and subword $w_1 = ...a...$, then you can change $w_0$  to  $b^{\pm x \mp 1 }a b^{\pm x \mp 1}$ and transform $w_1$ with $...b^{\pm}ab^{\pm}...$ at the same moment.
    
    \item[3. ] Clipping. Let $w$ represent element $(k,m,n)$.  Let $ABCD$ be rectangle such that: \\
    \begin{itemize}
        \item  $ k < 0 \And m >0 $ then $  A(n,k+m)$, $B(0,k+m)$, $C(0,-k)$ and $D(n,-k)$. \\ 
        \item $ k < 0 \And m < 0 $ then $ A(n,m-k)$, $B(0,m-k)$, $C(0,m+k)$ and $D(n,m+k)$. \\
        \item $ k > 0 \And m >0 $ then $ A(n,k+m)$, $B(0,k+m)$, $C(0,-k)$ and $D(n,-k)$. \\ 
       \item $ k > 0 \And m < 0 $ then $ A(n,m)$, $B(0,m)$, $C(0,k+m)$ and $D(n,k+m)$.\\  
       \item $ k < m  $ then $ A(n,0)$, $B(0,0)$, $C(0,m)$ and $D(n,m)$.\\  
    \end{itemize}
  




Then the polyline cuts the square into two ordered set of Young Diagrams (such that none of the Young Diagrams are sub-diagram of another one). If we can add or delete two squares of degree 2 in Young Diagrams such that they're still Young Diagrams and the squares in the odd distance (difference in first coordinates of left down corners), then we apply such change of subwords ($ab\mapsto ba$, for example), to obtain word $w'$ whose complement in $ABCD$ will be it's changed ordered set of Young Diagrams. Clipping is transformation from $w$ to $w'$.
   \\

\begin{figure}[h]
\begin{minipage}{.40\textwidth}
  \centering
  \includegraphics[scale = 0.6]{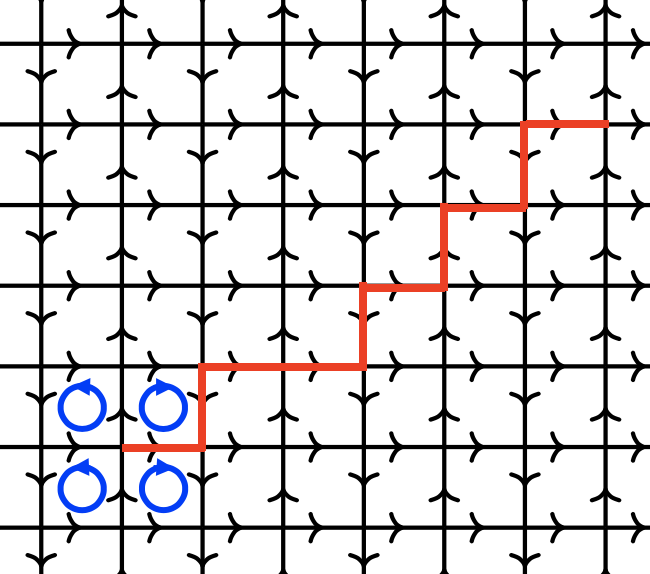}
  \end{minipage}
  \begin{minipage}{.05\textwidth}
  \centering
  \includegraphics[scale = 0.2]{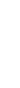}
\end{minipage}
  \begin{minipage}{.40\textwidth}
  \centering
  \includegraphics[scale = 0.6]{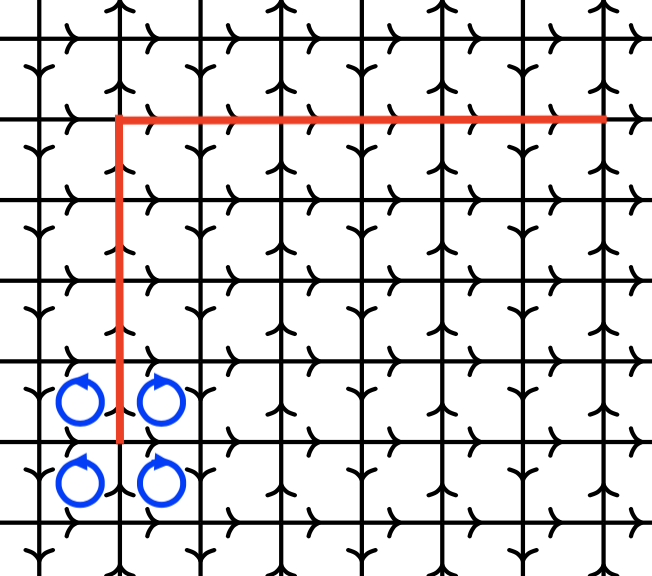}
  \end{minipage}
 \end{figure}

\begin{figure}[h]
\begin{minipage}{.20\textwidth}
  \centering
  \includegraphics[scale = 0.255]{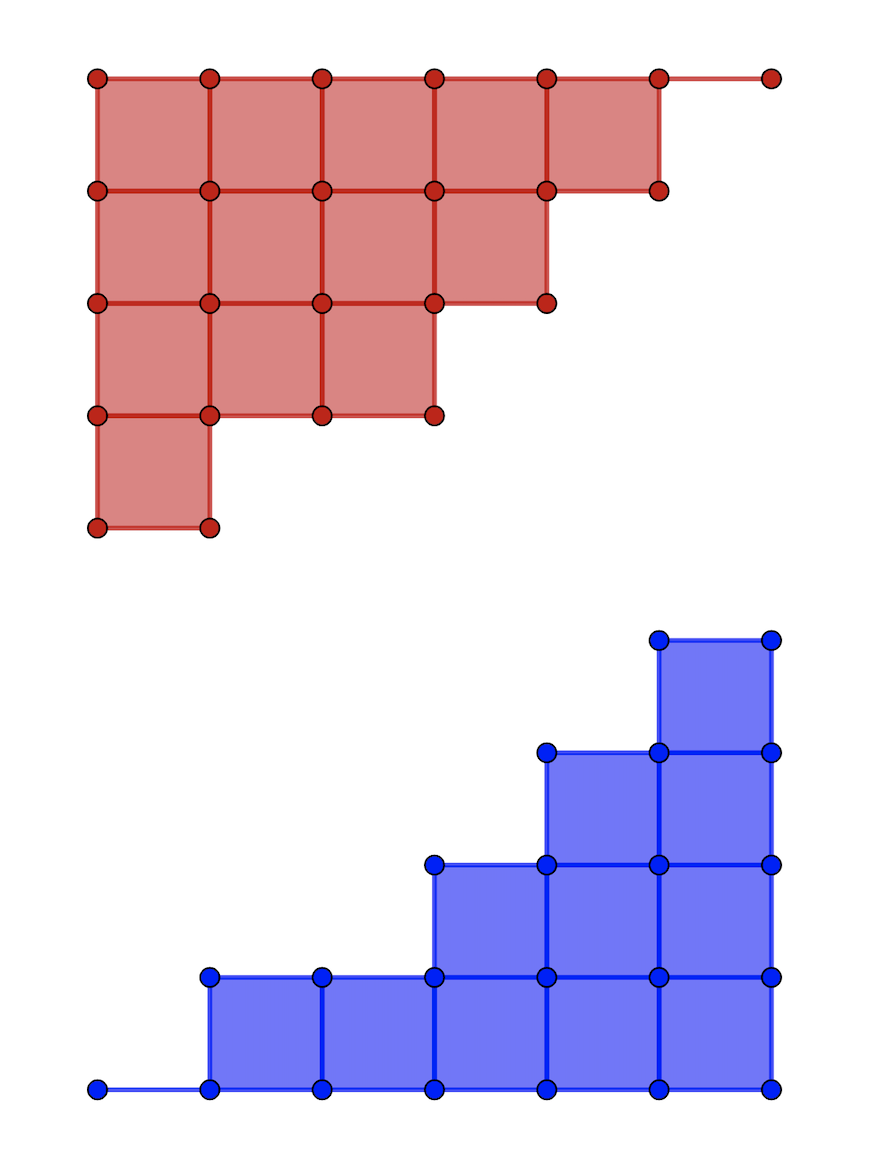}
\end{minipage}
\begin{minipage}{.05\textwidth}
  \centering
  \includegraphics[scale = 0.2]{Screenshot2.png}
\end{minipage}
\begin{minipage}{.20\textwidth}
  \centering
  \includegraphics[scale = 0.2]{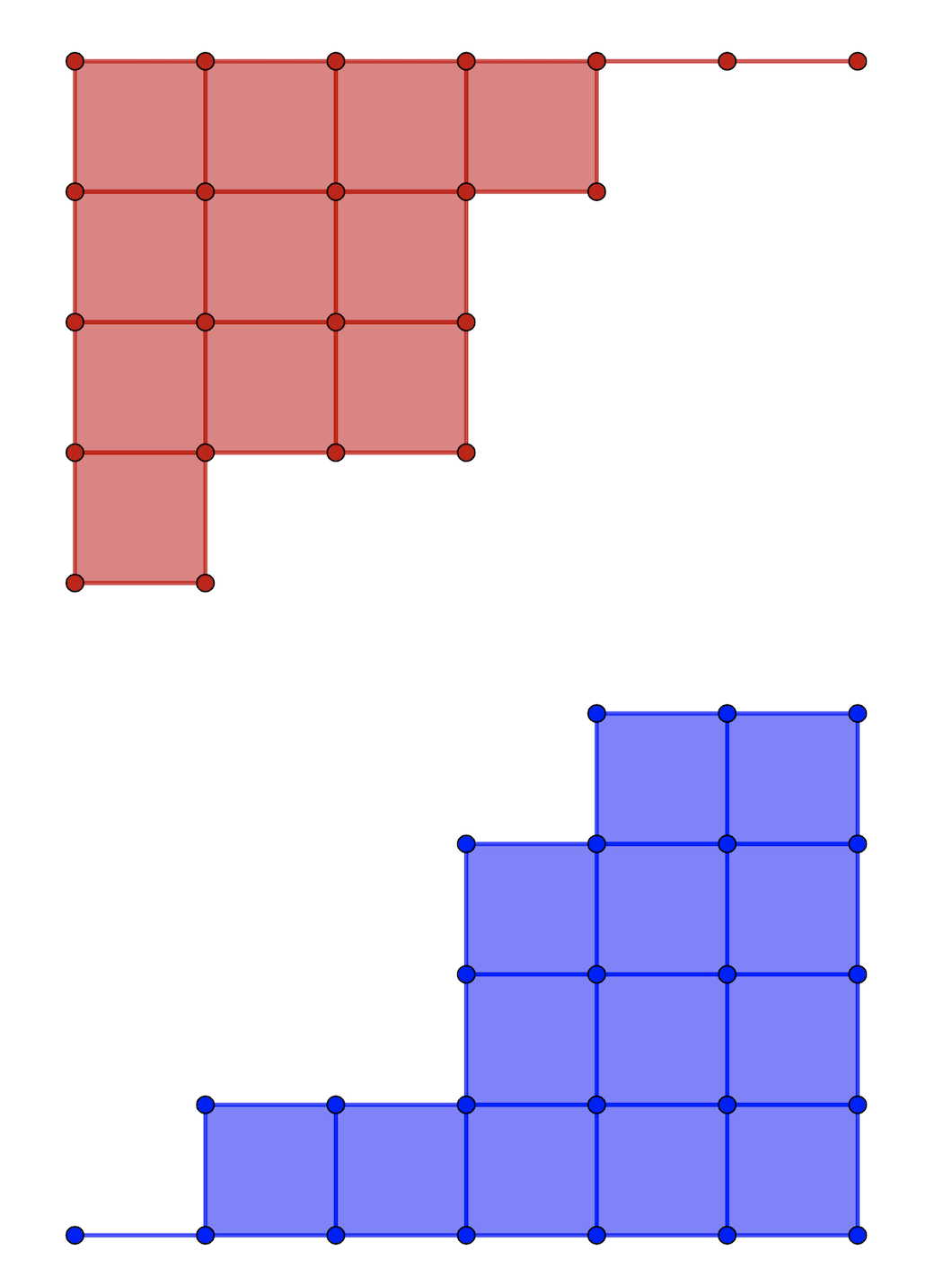}
\end{minipage}
\begin{minipage}{.05\textwidth}
  \centering
  \includegraphics[scale = 0.2]{Screenshot2.png}
\end{minipage}
\begin{minipage}{.20\textwidth}
  \centering
 \includegraphics[scale = 0.237]{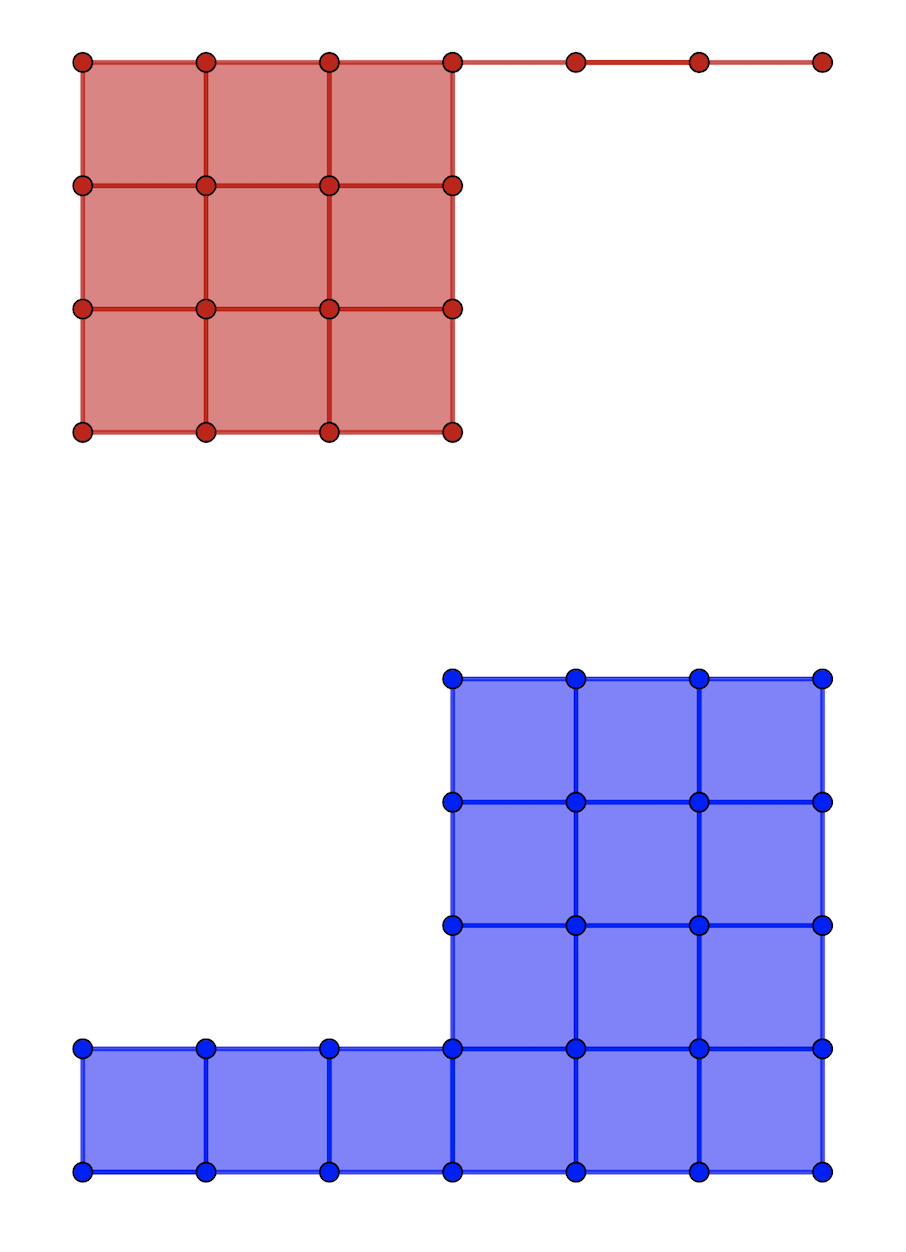}
\end{minipage}
\begin{minipage}{.05\textwidth}
  \centering
  \includegraphics[scale = 0.2]{Screenshot2.png}
\end{minipage}
\begin{minipage}{.20\textwidth}
  \centering
 \includegraphics[scale = 0.245]{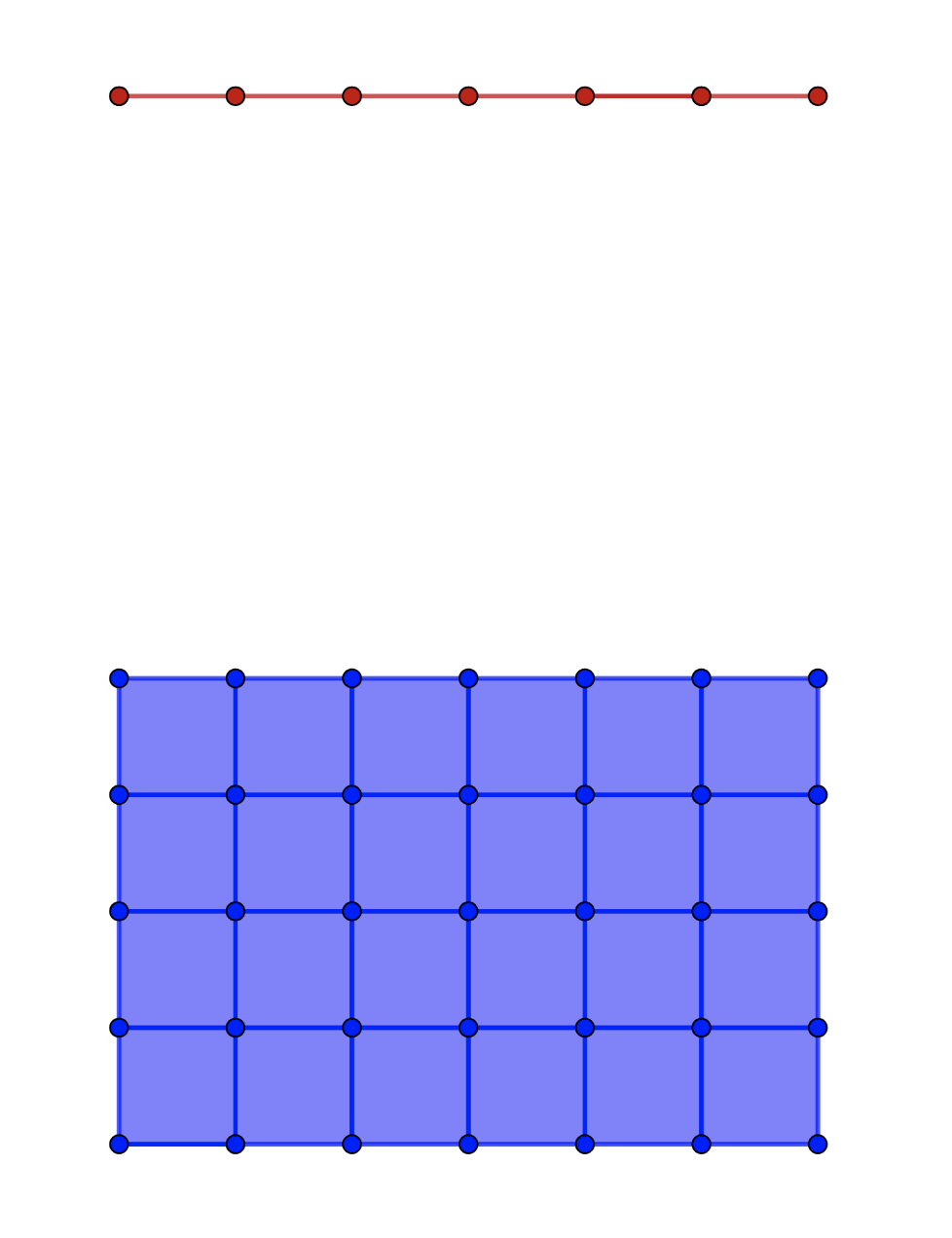}
\end{minipage}

\caption{Example of transforming a poly-line with $k<m$ into the standart representative}
\end{figure}







\begin{figure}[h]
\begin{minipage}{.49\textwidth}
  \centering
  \includegraphics[scale = 0.3]{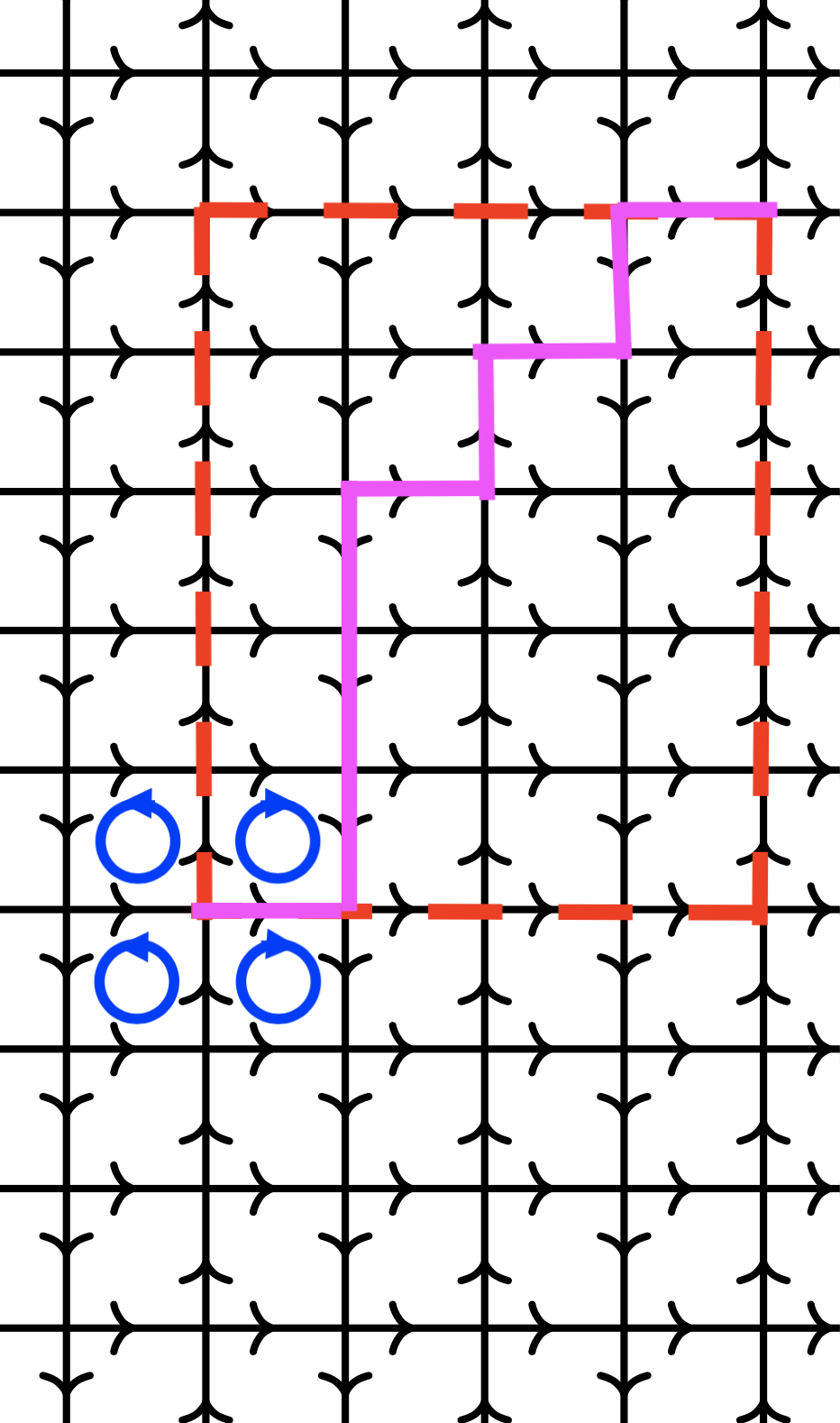}
\end{minipage}
\begin{minipage}{.49\textwidth}
  \centering
 \includegraphics[scale = 0.3]{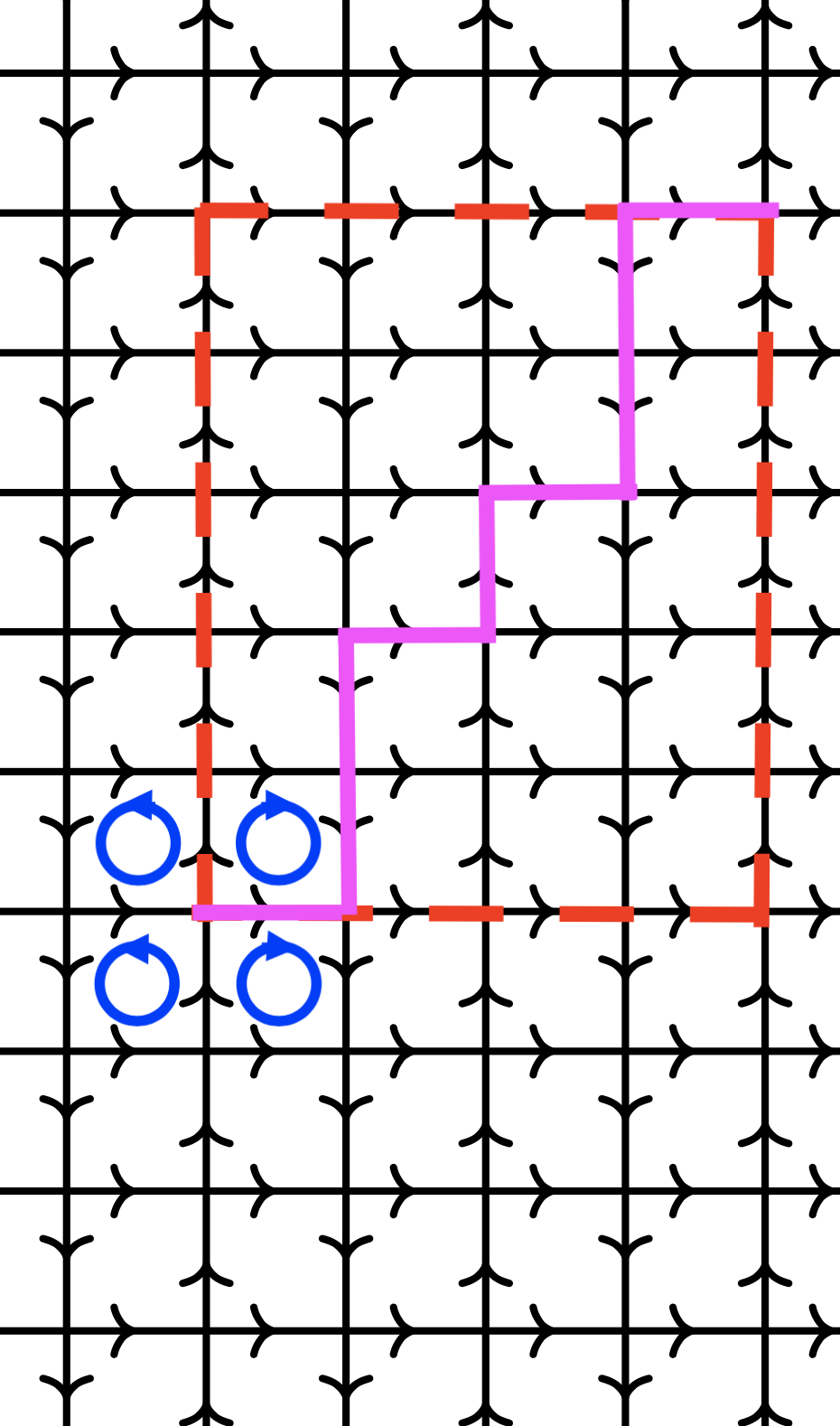}
\end{minipage}
\caption{Clipping for geodesic representative for $(-4,5,4)$}
\end{figure}

\end{enumerate}

\end{Definition}





\begin{figure}[h]
\center{\includegraphics[scale=0.2]{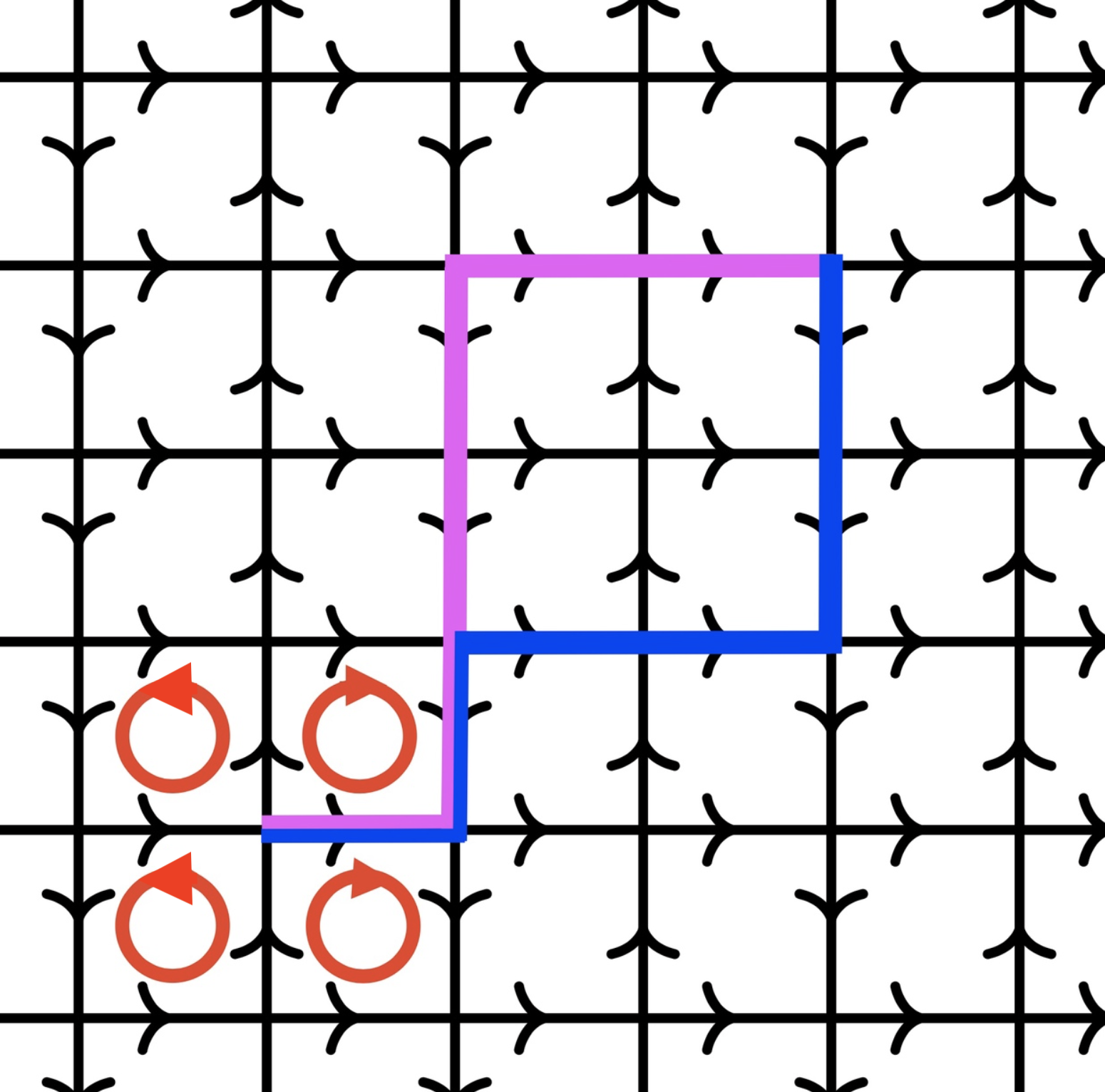}}
\caption{Even castling for word $ab^{-1}a^2b^{-2} = ab^{-3}a^2$}
\label{fig:image}
\end{figure}

\begin{figure}[h]
\begin{minipage}{.49\textwidth}
  \centering
  \includegraphics[scale = 0.2]{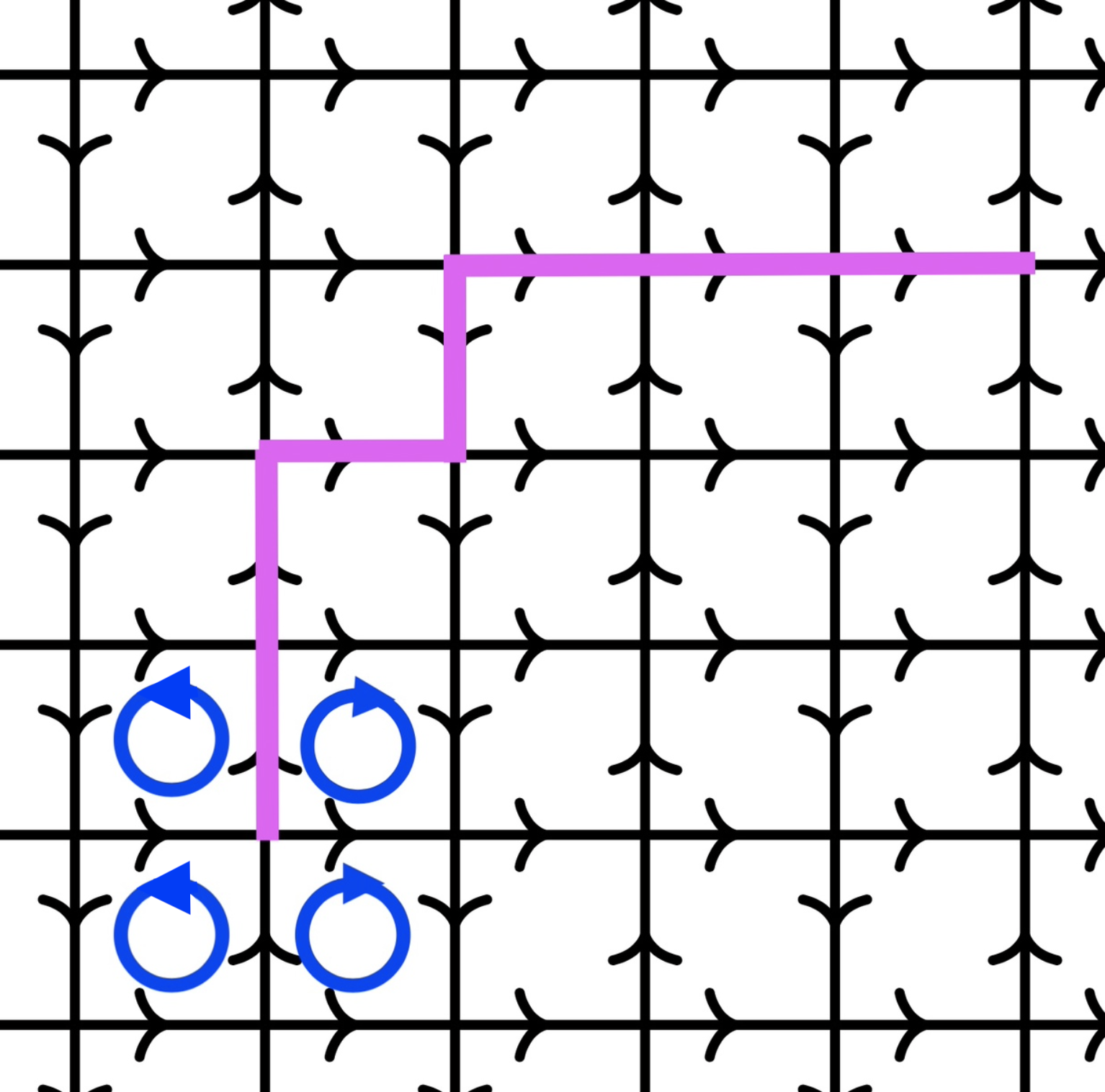}
\end{minipage}
\begin{minipage}{.49\textwidth}
  \centering
 \includegraphics[scale = 0.2]{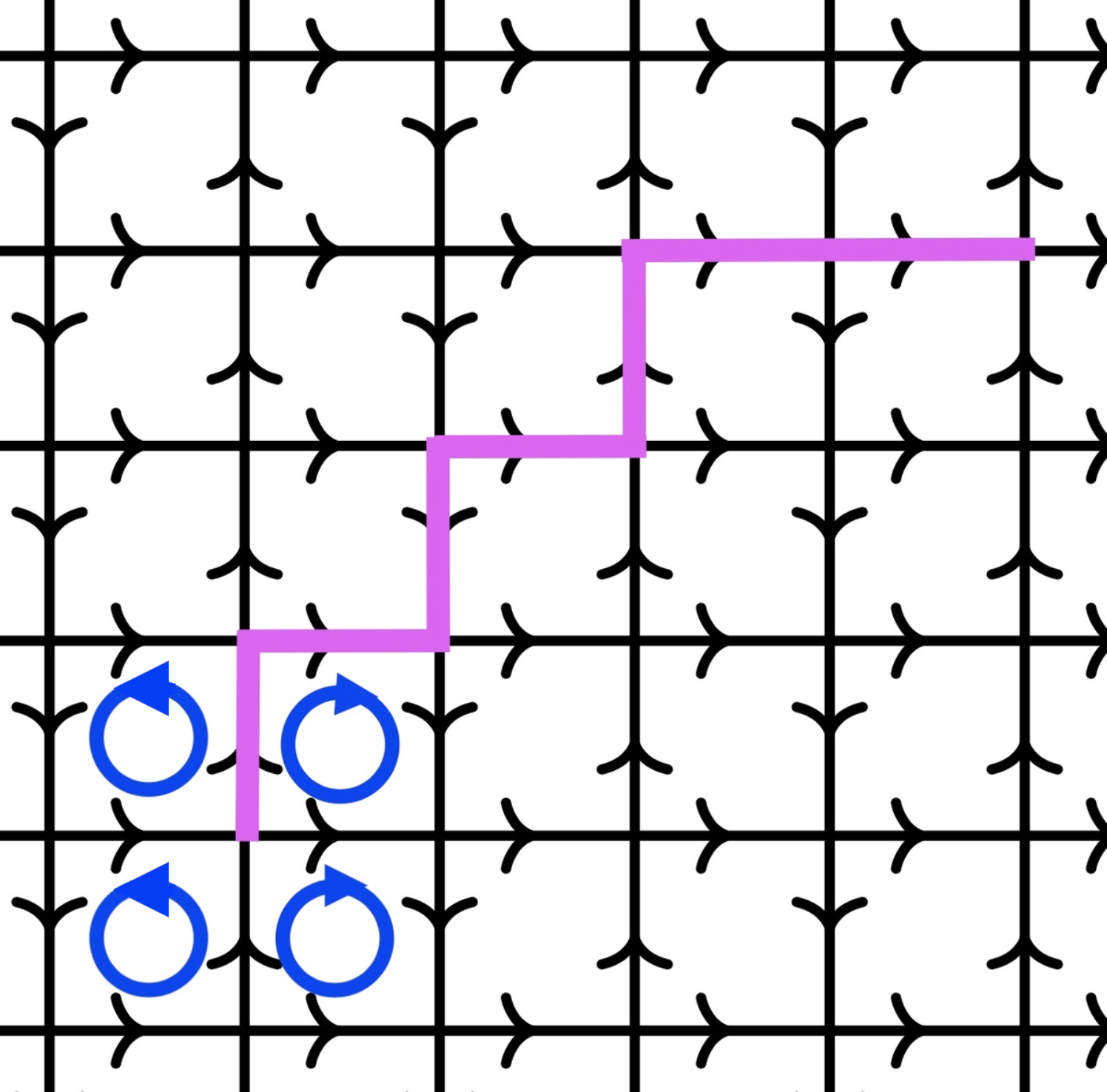}
\end{minipage}
\caption{Clipping for two geodesic representatives for $(-1,3,4)$}
\end{figure}

\begin{figure}[h]
\begin{minipage}{.49\textwidth}
  \centering
  \includegraphics[scale = 0.2]{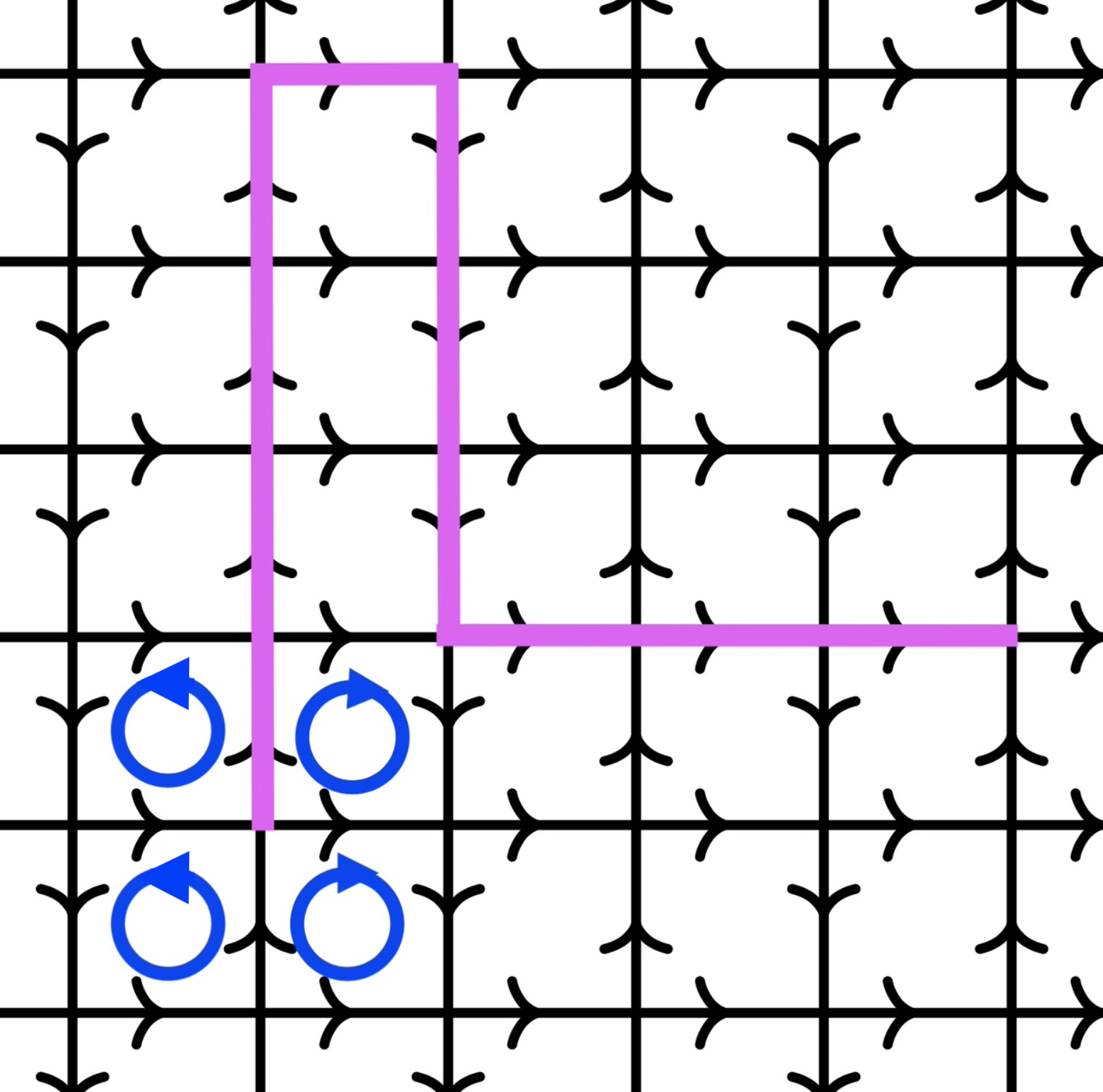}
\end{minipage}
\begin{minipage}{.49\textwidth}
  \centering
 \includegraphics[scale = 0.2]{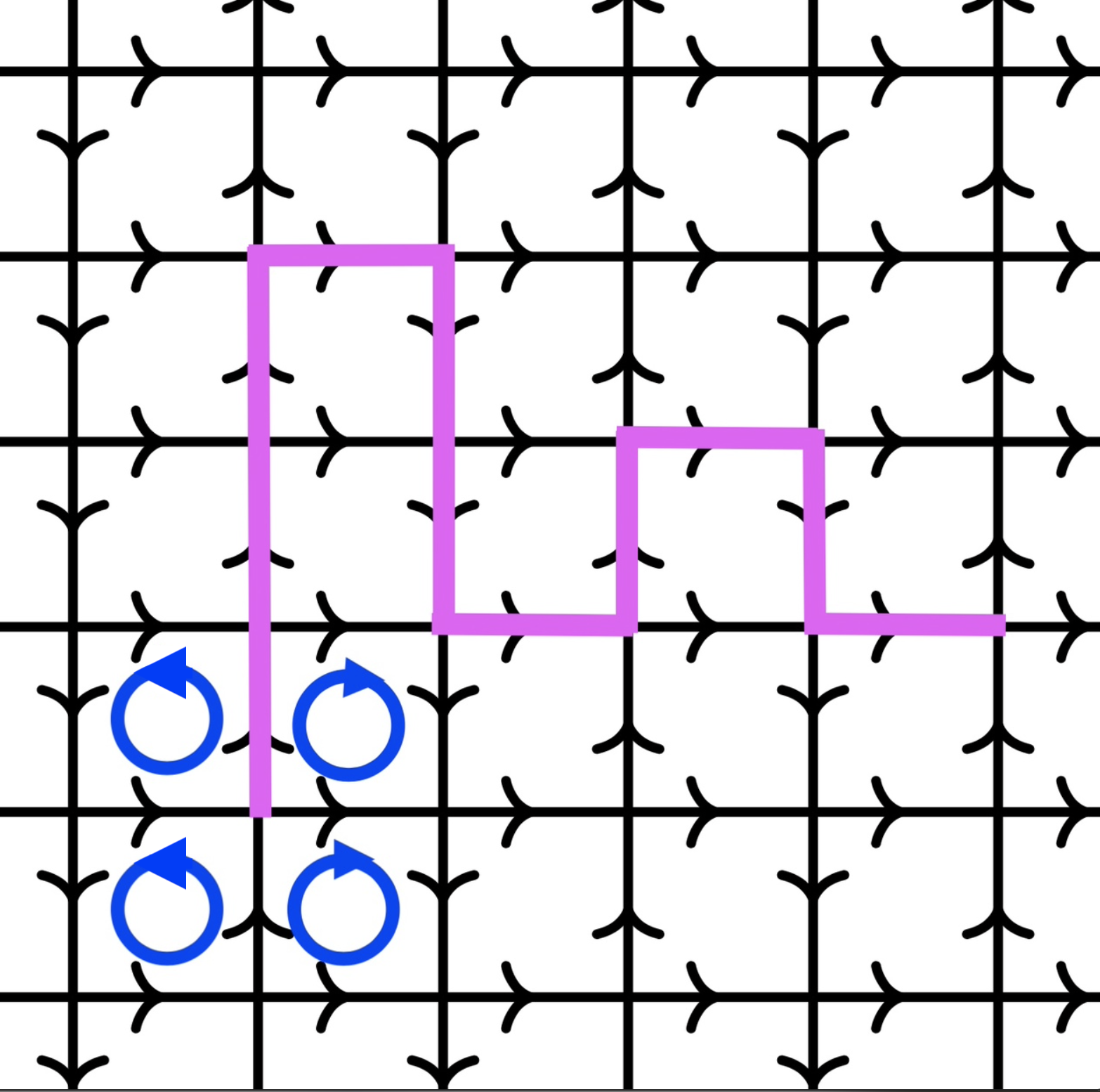}
\end{minipage}
\caption{Detowering for geodesic representatives of $(3,1,4)$}
\end{figure}



\begin{Lemma}
Two different polygonal lines have different sets of Young Diagrams which are obtained by applying basic moves. 
\end{Lemma}

\begin{proof}


It's obvious that these squares must be the same width. Then we will prove that since the polygonal lines are the same, then the sets of Young Diagrams are also the same. Suppose they are not the same. Then somewhere there should be an additional/missing square. But if there is then we can reconstruct the polylines from the diagrams and they wouldn't be the same. The proof in the other way is analogical.
\end{proof}


\begin{Lemma}
All basic moves don't change the element that is represented by the geodesic word and don't generate non-geodesic representatives by applying them.
\end{Lemma}

\begin{proof}

The first and second moves obviously don't change the length, area, and end-point of words because of lemmas in the previous subsection. The third move doesn't change area, because we add or delete such cells that the sum of their areas is zero (because of the distance). The length doesn't change in that case, because all such additions and deletions are realized as reordering pairs of letters with changing of sign in the case without creating subwords like $uu^{-1}$.

\end{proof}

The following Theorem is the purpose, why basic moves were introduced.

\begin{Theorem}\label{ll}
Any geodesic representative of element $(k,m,n)$ can be obtained from a standard geodesic representative by finitely many applying basic moves.
\end{Theorem}

\begin{proof}
Any geodesic word is uniquely presented by the union of Young Diagrams and rectangle (as in \ref{kok}, third basic move). We fix the rectangle for the standard representative and will transform it into the arbitrary geodesic representative. It's easy to understand that any Young Diagram can be obtained with 2 basic moves( clipping and even castlings). If we come up with a situation in that our geodesic representative has a different amount of Young Diagrams we use detowering to transform it to the standard representative.
\end{proof}







\newpage

\begin{figure}[h]
\begin{minipage}{.40\textwidth}
  \centering
  \includegraphics[scale = 0.4]{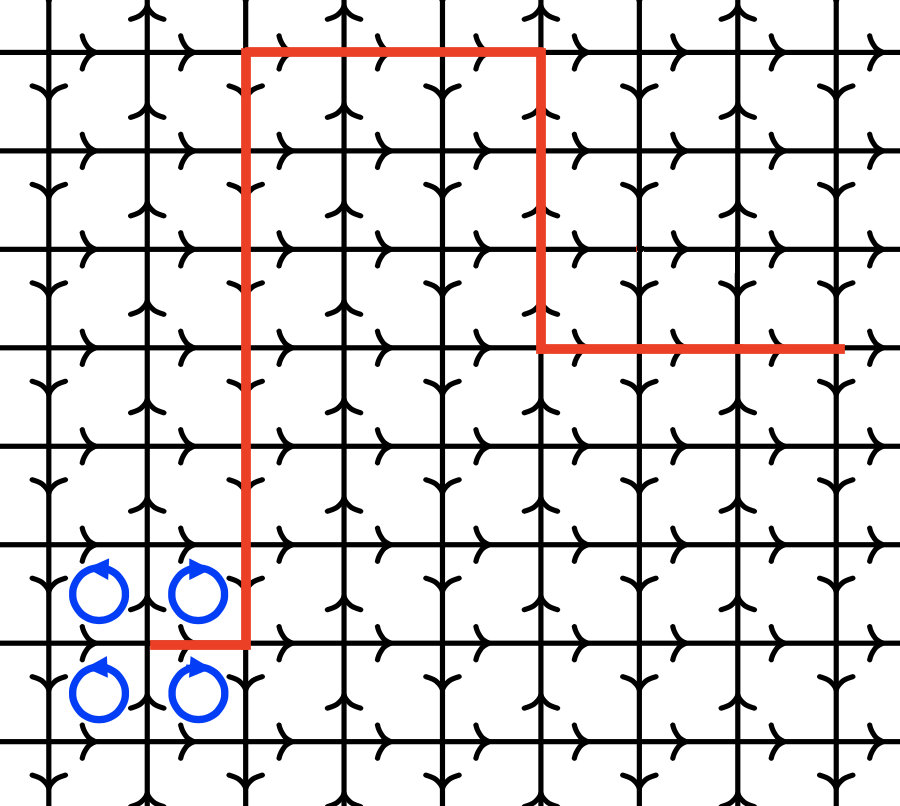}
  \end{minipage}
  \begin{minipage}{.05\textwidth}
  \centering
  \includegraphics[scale = 0.2]{Screenshot2.png}
\end{minipage}
  \begin{minipage}{.40\textwidth}
  \centering
  \includegraphics[scale = 0.4]{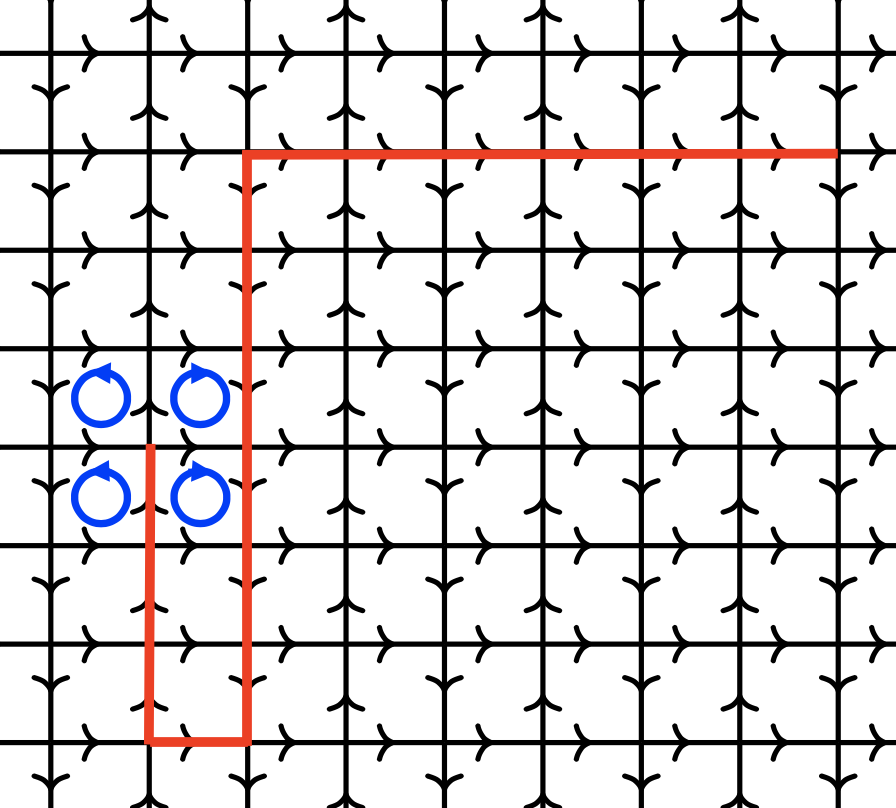}
  \end{minipage}
 \end{figure}

\newpage

\begin{figure}[h]
\begin{minipage}{.20\textwidth}
  \centering
  \includegraphics[scale = 0.5]{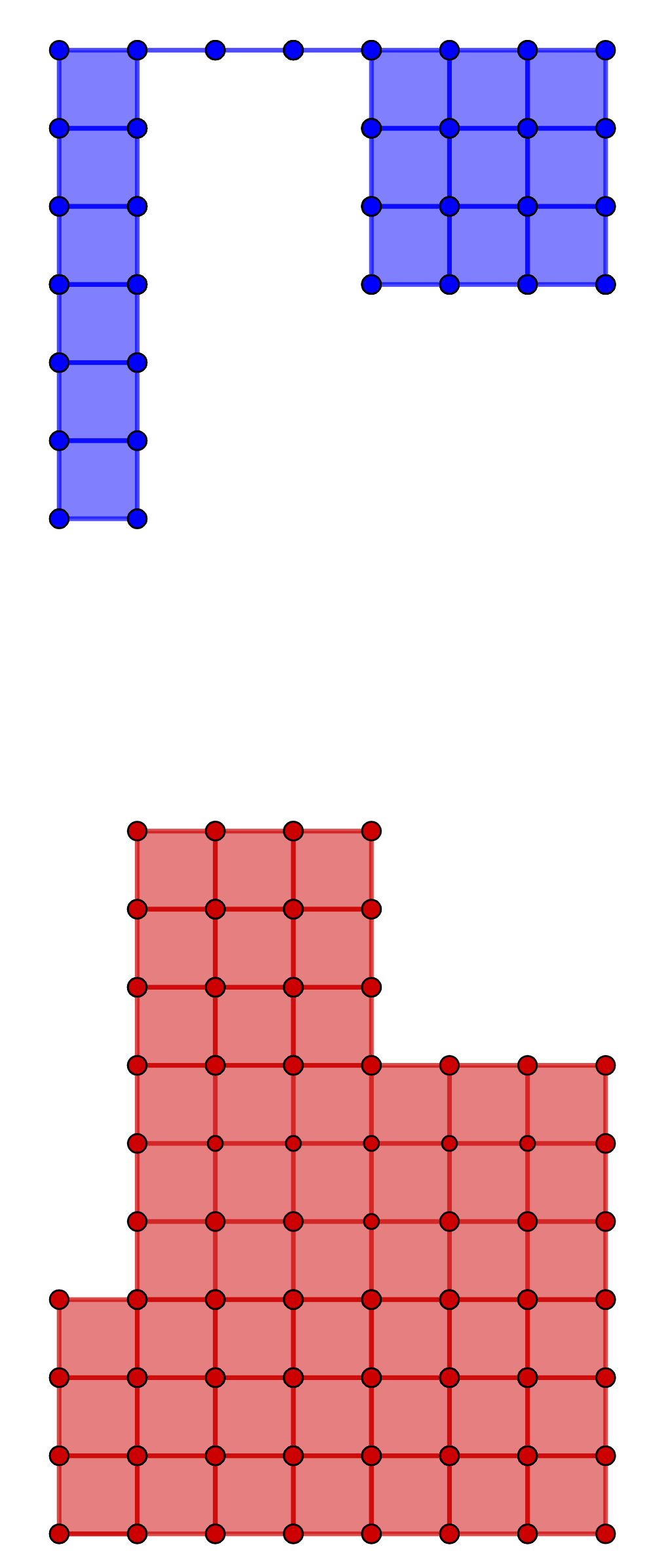}
\end{minipage}
\begin{minipage}{.05\textwidth}
  \centering
  \includegraphics[scale = 0.2]{Screenshot2.png}
\end{minipage}
\begin{minipage}{.20\textwidth}
  \centering
 \includegraphics[scale = 0.5]{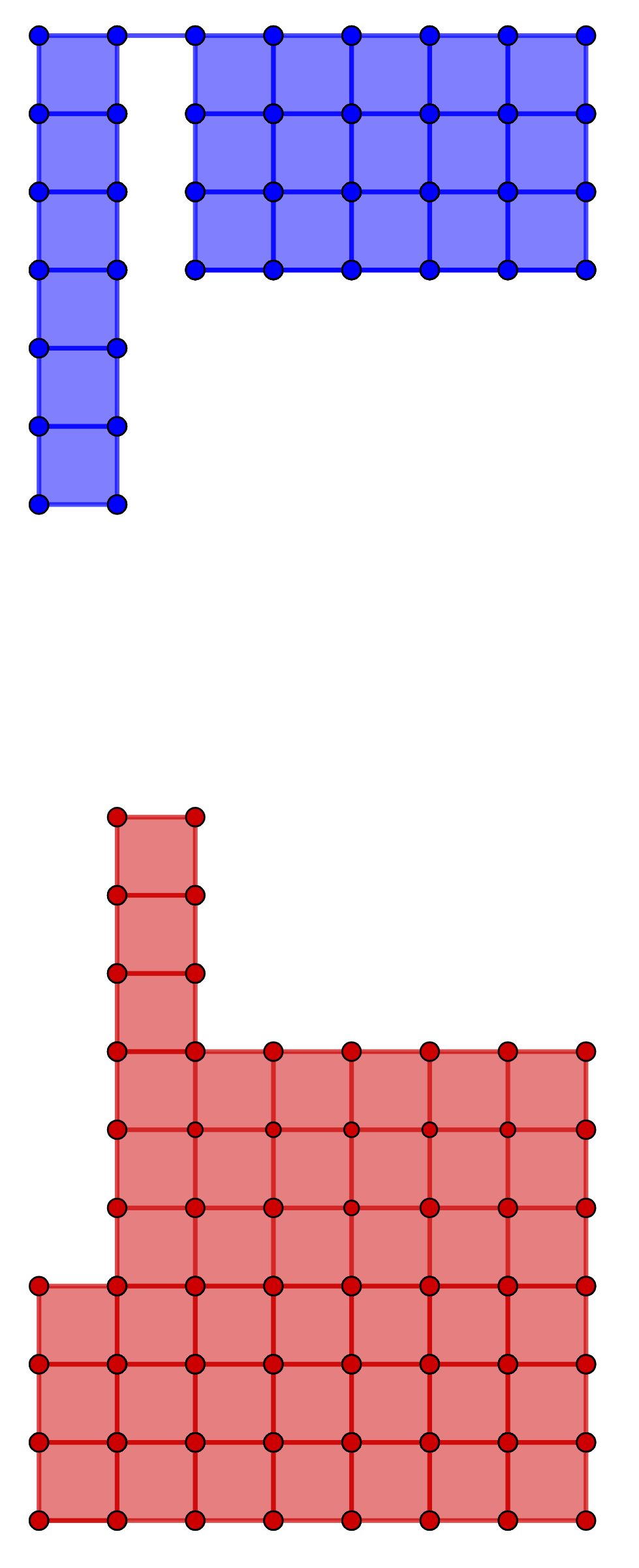}
\end{minipage}
\begin{minipage}{.05\textwidth}
  \centering
  \includegraphics[scale = 0.2]{Screenshot2.png}
\end{minipage}
\begin{minipage}{.20\textwidth}
  \centering
 \includegraphics[scale = 0.5]{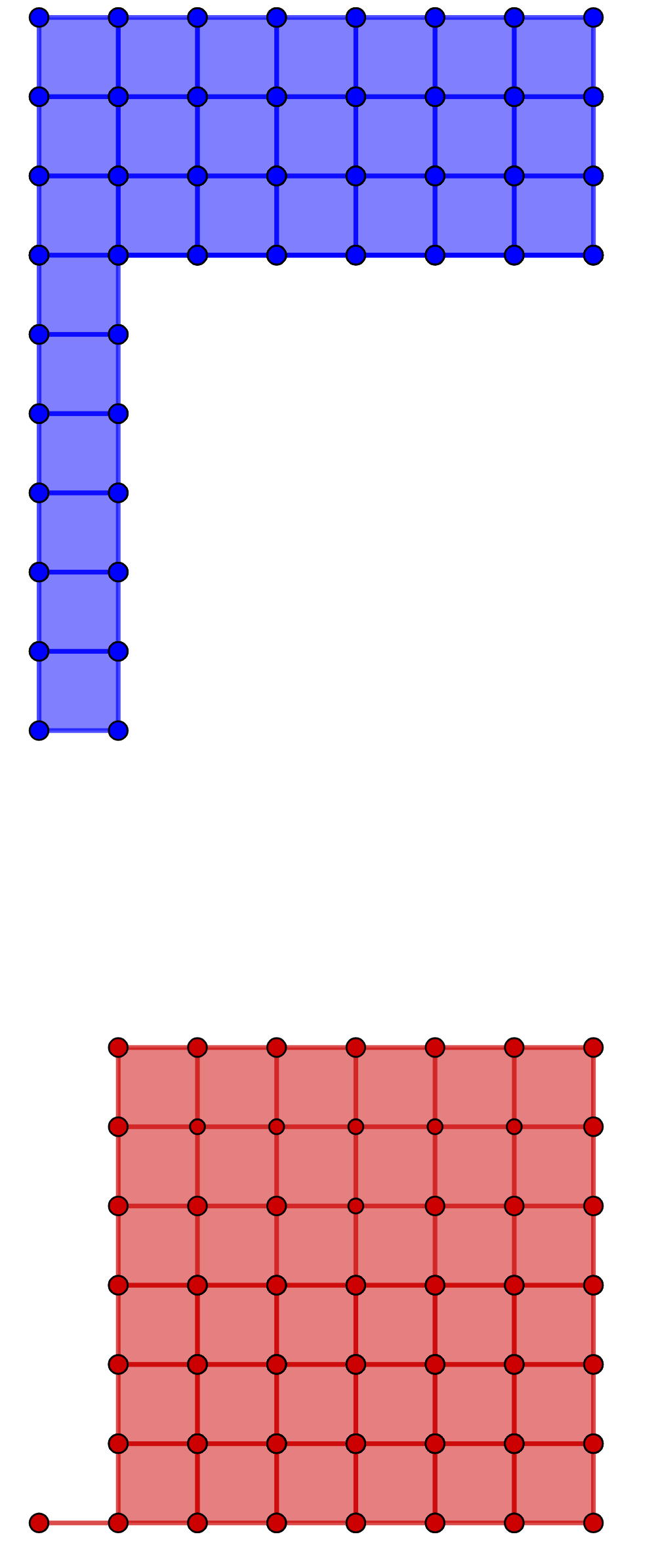}
\end{minipage}
\caption{Example of transforming a poly-line with $k>m$ into the standart representative}
\end{figure}

\vspace{3cm}

\end{document}